\documentclass[12pt]{amsart}
\usepackage{epsfig,amsmath,amsfonts,amssymb,amsthm}
\usepackage[all]{xy}
\usepackage{amscd}
\usepackage{amsmath, pb-diagram}
\usepackage{enumerate}
\usepackage{calrsfs}
\usepackage{accents}
\DeclareMathAlphabet{\pazocal}{OMS}{zplm}{m}{n}
\newtheorem{thm}{Theorem}[section]
\newtheorem{lem}[thm]{Lemma}
\newtheorem{prop}[thm]{Proposition}
\newtheorem{cor}[thm]{Corollary}
\newtheorem{defn}[thm]{Definition}

\newtheorem{rem}[thm]{Remark}
\newtheorem{example}[thm]{Example}

\begin{document}

\title[(Strongly) $M-\pazocal{A}-$Injective(Flat) Modules ]
{(Strongly) $M-\pazocal{A}-$Injective (Flat) Modules}

\author{Tah\.{i}re \" Ozen}
\address{Department of Mathematics, Abant \.{I}zzet Baysal University
\newline
G\"olk\"oy Kamp\"us\"u Bolu, Turkey} \email{ozen\_t@ibu.edu.tr}

\date{\today}
\subjclass{16D40, 16D50, 16P70}
\keywords{$\pazocal{A}-\mathrm{injective}$, $\pazocal{A}-\mathrm{flat}$, $\pazocal{A}$-coherent module, strongly injective and flat, $FS$ and $PS$ module, precover, preenvelope.}

\pagenumbering{arabic}

\begin{abstract}Let $M$ be a left $R-$module and $\pazocal{A}=\{A\}_{A\in\pazocal{A}}$ be a family of some submodules of $M$. It is introduced the classes of (strongly) $M-\pazocal{A}-\mathrm{injective}$ and (strongly) $M-\pazocal{A}-\mathrm{flat}$ modules which are denoted by $(S) M-\pazocal{A}I$ and $(S) M-\pazocal{A}F$, respectively. It is obtained some characterizations of these classes and the relationships between these classes. Moreover it is investigated $(S) M-\pazocal{A}I$ and $(S) M-\pazocal{A}F$ precovers and preenvelopes of modules. It is also studied $\pazocal{A}$-coherent, $F\pazocal{A}$ and $P\pazocal{A}$ modules. Finally more generally we give the characterization of $S-\pazocal{A}I(F)$ modules where $\pazocal{A}=\{A\}_{A\in\pazocal{A}}$ is a family of some left $R-$modules.
\end{abstract}
\maketitle
\section{introduction}
Let $M$ be a left $R-$module and $\pazocal{A}=\{A\}_{A\in\pazocal{A}}$ be a family of some submodules of $M$ where $R$ is an associative ring with identity and $M^{+}=Hom_{\mathbb{Z}}(M,Q/\mathbb{Z})$.

In this paper the classes of $M-\pazocal{A}-\mathrm{injective}$ and $M-\pazocal{A}-\mathrm{flat}$ modules are studied which are denoted by $M-\pazocal{A}I$ and $M-\pazocal{A}F$ and generalizations of relative injective, $\pazocal{A}-\mathrm{injective}$ and $\pazocal{A}-\mathrm{flat}$ modules (see \cite{2}, \cite{1}, \cite{5} and \cite{18}). As a special case, $M$-min-injective and $M$-min-flat modules are given in \cite{3}. Moreover we introduce the class of strongly $M-\pazocal{A}-\mathrm{injective}$ modules which is denoted by $S M-\pazocal{A}I$ and a generalization of strongly $M-\mathrm{injective}$ modules (see \cite{4}) and we introduce the class of strongly $M-\pazocal{A}-\mathrm{flat}$ modules which is denoted by $S M-\pazocal{A}F$. Thus we have that $I(F)\subseteq S M-\pazocal{A}I(F)\subseteq M-\pazocal{A}I(F)$ where $I(F)$ is the class of all injective (flat) modules.

We prove that $N$ is in $S M-\pazocal{A}F$ if and only if $N^{+}$ is in $S M-\pazocal{A}I$ and if $M$ is finitely presented and $\pazocal{A}-coherent$, then $N$ is in $S M-\pazocal{A}I$ if and only if $N^{+}$ is in $S M-\pazocal{A}F$. We obtain that $S M-\pazocal{A}I$ and $S M-\pazocal{A}F$ have some properties which are equivalent to the condition that $M$ is finitely presented and $\pazocal{A}-coherent$. We have that $(S) M-\pazocal{A}F$ is a Kaplansky class and if $M$ is finitely presented and $\pazocal{A}-\mathrm{coherent}$, then $(S) M-\pazocal{A}I$ is also a Kaplansky class. Moreover if $M$ is finitely presented and $\pazocal{A}-coherent$, then every left $R-module$ has $S M-\pazocal{A}I$ preenvelope and cover and we have that $((S) M-\pazocal{A}F(I),((S) M-\pazocal{A}F(I))^\bot)$ is a perfect cotorsion theory (where $M$ is finitely presented, $\pazocal{A}-\mathrm{coherent}$ and $R$ is in $(S) M-\pazocal{A}I$) and the cotorsion theory $(^{\perp}(S M-\pazocal{A}I),S M-\pazocal{A}I)$ is complete. We also study $F \pazocal{A}$ and $P \pazocal{A}$ modules as generalizations of $FS$ and $PS$ modules (see \cite{19},\cite{20} and \cite{21}). Finally we give some properties of the class  
$S-\pazocal{A}I(F)$  where $\pazocal{A}=\{A\}_{A\in\pazocal{A}}$ is a family of some left $R-$modules which is a generalization of the class $SM-\pazocal{A}I(F)$.

\maketitle

\section{a generalization of relative injective and flat modules}

We begin with giving the following definition:

\begin{defn}
A left $R-$module $T$ is called $M-\pazocal{A}-\mathrm{injective}$ if for every exact sequence $0\rightarrow A\rightarrow M\rightarrow M/A\rightarrow0$, we have an exact sequence $Hom(M,T)\rightarrow Hom(A,T)\rightarrow0$ for all $A\in\pazocal{A}$. $M-\pazocal{A}-\mathrm{injective}$ modules are generalization of $M-\mathrm{injective}$ modules \emph{(}see \cite{2} for more detail\emph{)}. The class of $M-\pazocal{A}-\mathrm{injective}$ modules is denoted by $M-\pazocal{A}I$.
\end{defn}
\begin{example}
Let $M={_{R}R}$ and $\pazocal{A}$ be a family of some ideals of $R$. Then an $R-\pazocal{A}-\mathrm{injective}$ module is taken as $\pazocal{A}-$injective \emph{(}see \cite{1}\emph{)}.
\end{example}
Moreover let $M$ be any left $R-$module and $\pazocal{A}$ be a family of all simple submodules. Then $M-\pazocal{A}-\mathrm{injective}$ is called $M-\mathrm{min}-\mathrm{injective}$ in \cite{3}.
\begin{defn}
A left $R-$module $T$ is called a strongly $M-\pazocal{A}-\mathrm{injective}$ if for every exact sequence $0\rightarrow X\rightarrow Y\rightarrow Y/X\rightarrow0$ where $Y/X\cong M/A$ for some $A\in\pazocal{A}$, we have an exact sequence $Hom(Y,T)\rightarrow Hom(X,T)\rightarrow0$, or equivalently $Ext^{1}_{R}(M/A,T)=0$ for all $A\in\pazocal{A}$. The definition is generalization of strongly $M-\mathrm{injective}$ modules \emph{(}see \cite{4}\emph{)}.
\end{defn}
The class of strongly $M-\pazocal{A}-\mathrm{injective}$ modules is denoted by $S M-\pazocal{A}I$.
\begin{rem}
Let I denote the class of injective modules. Then we have that $$I\subseteq S M-\pazocal{A}I\subseteq M-\pazocal{A}I.$$
\end{rem}
Let $M$ be projective (in particular, let $M={_{R}R}$). Then strongly $M-\pazocal{A}-\mathrm{injective}$  modules and $M-\pazocal{A}-\mathrm{injective}$ modules are identical.\par
Now, we explain that the inclusions can be proper with respect to $M$.\par
For example, take $M=\mathbb{Z}_{6}$ and $\pazocal{A}=\{<\overline{2}>,<\overline{3}>\}$. Then $M=\mathbb{Z}_{6}=<\overline{2}>\oplus<\overline{3}>$ and so every $\mathbb{Z}-$module is $\mathbb{Z}_{6}-\pazocal{A}-$injective. But $\mathbb{Z}_{2}$ is not $S \mathbb{Z}_{6}-\pazocal{A}-$injective.\par
Let $R$ be a ring and $_{R}R=I_{1}\oplus I_{2}$ where $I_{1}$ and $I_{2}$ are left ideals of $R$ and $\pazocal{A}=\{I_{1},I_{2}\}$. Then every left $R-$module is $S R-\pazocal{A}-$injective, but it is not necessary that every left $R-$module is injective.
\begin{defn}
A right $R-$module $T$ is called $M-\pazocal{A}-\mathrm{flat}$ if for every exact sequence $0\rightarrow A\rightarrow M\rightarrow0\rightarrow M/A\rightarrow0$, we have an exact sequence $0\rightarrow T\otimes A\rightarrow T\otimes M$.
\end{defn}
The class of $M-\pazocal{A}-\mathrm{flat}$ modules is denoted by $M-\pazocal{A}F$.
\begin{example}
Let $M={_{R}R}$ and $\pazocal{A}$ be a family of some ideals of $R$. Then $R-\pazocal{A}-\mathrm{flat}$ is taken as $\pazocal{A}-\mathrm{flat}$ \emph{(}see \cite{5}\emph{)}.
\end{example}
Moreover let $M$ be any left $R-$module and $\pazocal{A}$ be a family of simple submodules of $M$. Then $M-\pazocal{A}-\mathrm{flat}$ is called $M-\mathrm{min}-\mathrm{flat}$ in \cite{3}.
\begin{lem}\label{lem7}
The followings are equivalent:\par
\begin{itemize}
  \item[(i)] Let $T$ be a right $R-$module. For every exact sequence $0\rightarrow X\rightarrow Y\rightarrow Y/X\rightarrow0$ where $Y/X\cong M/A$ for some $A\in\pazocal{A}$, we have an exact sequence $0\rightarrow T\otimes X\rightarrow T\otimes Y$ is exact.
  \item[(ii)] $Tor_{1}^{R}(T,M/A)=0$.
\end{itemize}
\end{lem}
\begin{proof}
(ii)$\Rightarrow$(i) is easy. (i)$\Rightarrow$(ii), let $0\rightarrow T\otimes X\rightarrow T\otimes Y\rightarrow T\otimes Y/X\rightarrow0$ be exact. Then $(T\otimes Y)^{+}\rightarrow(T\otimes X)^{+}\rightarrow0$ is exact and hence $Hom(Y,T^{+})\rightarrow Hom(X,T^{+})\rightarrow0$ is exact. Thus $T^+$ is strongly $M-\pazocal{A}-\mathrm{injective}$ and then $Ext^{1}_{R}(M/A,T^+)=0$. Since $Ext^{1}_{R}(M/A,T^+)\cong Hom_{\mathbb {Z}}(Tor_{1}^{R}(T,M/A),\mathbb{Q}/\mathbb{Z})$ and $\mathbb{Q}/\mathbb{Z}$ is an injective cogenerator, $Tor_{1}^{R}=(T,M/A)=0.$
\end{proof}
\begin{defn}
A right $R-$module satisfying the equivalent conditions in Lemma ~\ref{lem7} is called a strongly $M-\pazocal{A}-\mathrm{flat}$ module.\par
The class of strongly $M-\pazocal{A}-\mathrm{flat}$ modules is denoted by $S M-\pazocal{A}F$.
\end{defn}
\begin{rem}
Let $F$ denote the class of flat right $R-$modules. Then we have that $$F\subseteq S M-\pazocal{A}F\subseteq M-\pazocal{A}F.$$ Let $M$ be flat \emph{(}in particular, let $M={_{R}R}$\emph{)} then $S M-\pazocal{A}F=M-\pazocal{A}F$.
\end{rem}
Now, we explain that the inclusions can be proper with respect to $M$. Let $M=\mathbb{Z}_{6}$ and $\pazocal{A}=\{<\overline{2}>,<\overline{3}>\}$. Then $\mathbb{Z}_{2}^+$ is $\mathbb{Z}_{6}-\pazocal{A}-\mathrm{flat}$ which is in Lemma ~\ref{lem14}. But $\mathbb{Z}_{2}^+$ is not $S \mathbb{Z}_{6}-\pazocal{A}-\mathrm{flat}$.
\begin{lem}\label{lem10}
\begin{itemize}
\item[(i)] $(S) M-\pazocal{A}I$ is closed under (extensions and) direct summand and direct product.
\item[(ii)] $(S) M-\pazocal{A}F$ is closed under (extensions and) direct summand, direct sum and direct limit.
\end{itemize}
\end{lem}
\begin{proof}
It follows from definitions.
\end{proof}
\begin{lem}\label{lem11}
A right $R-$module $N$ is (strongly) $M-\pazocal{A}-\mathrm{flat}$ if and only if $N^+$ is (strongly) $M-\pazocal{A}-\mathrm{injective}$.
\end{lem}
\begin{proof}
We give just the proof of that $N$ is $M-\pazocal{A}-\mathrm{flat}$ if and only if $N^+$ is $M-\pazocal{A}-\mathrm{injective}$. Similarly we can prove the other part.\par
Let $A\in\pazocal{A}$. Then $0\rightarrow N\otimes A\rightarrow N\otimes M$ is exact if and only if $(N\otimes M)^{+}\rightarrow (N\otimes A)^{+}\rightarrow0$ is exact if and only if $Hom(M,N^+)\rightarrow Hom(A,N^+)\rightarrow0$ is exact. So $N$ is $M-\pazocal{A}-\mathrm{flat}$ if and only if $N^+$ is $M-\pazocal{A}-\mathrm{injective}$.
\end{proof}
\begin{defn}
$M$ is called (quotient) $\pazocal{A}-\mathrm{coherent}$ if for all $A\in\pazocal{A}$ ($M/A$) $A$ is finitely presented.
\end{defn}
\begin{lem}\label{lem13}
Let $M$ be finitely presented and $\pazocal{A}-\mathrm{coherent}$. Then $M$ is quotient $\pazocal{A}-\mathrm{coherent}$.
\end{lem}
\begin{proof}
Let $A\in\pazocal{A}$. We have the diagram
\[\begin{diagram}
\node{(\Pi R)\otimes A} \arrow{e}\arrow{s,t}{\alpha} \node{(\Pi R)\otimes M} \arrow{e} \arrow{s,t}{\beta} \node{(\Pi R)\otimes M/A} \arrow{e} \arrow{s,t}{\gamma} \node{0} \\
\node{\Pi A} \arrow{e} \node{\Pi M} \arrow{e} \node{\Pi M/A} \arrow{e} \node{0}
\end{diagram}\]\par
By Lemma 3.2.22 in \cite{8}, $\alpha$ and $\beta$ are isomorphism. By five lemma $\gamma$ is also isomorphism. By Theorem 3.2.33 in \cite{8}, $M/A$ is finitely presented.
\end{proof}
\begin{lem}\label{lem14}
Let $M$ be finitely presented left $R-\mathrm{module}$ and $\pazocal{A}-\mathrm{coherent}$. Then the following are equivalent:\par
\begin{itemize}
\item[(i)]$N$ is in $(S) M-\pazocal{A}I$.
\item[(ii)]$N^+$ is in $(S) M-\pazocal{A}F$.
\end{itemize}
Notice that it is  not necessary to be that $M$ is finitely presented to prove that $N$ is in $M-\pazocal{A}I$ implies that $N^+$ is in $M-\pazocal{A}F$.
\end{lem}
\begin{proof}
(i)$\Leftrightarrow$(ii). Since $M$ is quotient $\pazocal{A}-\mathrm{coherent}$ by Lemma ~\ref{lem13}, by Lemma 3.60 in \cite{6} we have the isomorphism $$Tor_{1}^{R}(N^+,M/A)\cong Ext^{1}_{R}(M/A,N)^+ .$$ So $N$ is in  $S M-\pazocal{A}I$ if and only if $N^+$ is in $S M-\pazocal{A}F$. Now, we will prove that $N$ is in $M-\pazocal{A}I$  if and only if $N^+$ is $M-\pazocal{A}F$. We have the following commutative diagram as follow:\par
\[\begin{diagram}
\node{N^+\otimes A} \arrow{e,t}{\alpha}\arrow{s,t}{\beta} \node{N^+\otimes M} \arrow{s,t}{\theta} \\
\node{Hom(A,N)^+} \arrow{e,t}{\gamma} \node{Hom(M,N)^+}
\end{diagram}\]\par
Since $A$ and $M$ are finitely presented, by Lemma 3.60 in \cite{6} $\beta$ and $\theta$ are isomorphism. So $N$ is in $M-\pazocal{A}I$ if and only if $N^+$ is in $M-\pazocal{A}F$.
\end{proof}
The following corollaries follow from Lemma ~\ref{lem11} and Lemma \ref{lem14}.
\begin{cor}\label{cor15}
Let $M$ be finitely presented and $\pazocal{A}-\mathrm{coherent}$. The followings are equivalent:
\begin{itemize}
\item[(i)] Every $(S) M-\pazocal{A}I$ is injective.
\item[(ii)] Every $(S) M-\pazocal{A}F$ is flat.
\end{itemize}
\end{cor}
\begin{cor}\label{cor16}
Let $M$ be finitely presented and $\pazocal{A}-\mathrm{coherent}$. The followings are equivalent:\par
\begin{itemize}
\item[(i)] Every left $R-$module is in $(S) M-\pazocal{A}I$.
\item[(ii)] Every right $R-$module is in $(S) M-\pazocal{A}F$.
\end{itemize}
\end{cor}
\begin{example}\label{prop17}
A simple module over a commutative ring is $SM-\pazocal{A}-\mathrm{flat}$ if and only if it is $SM-\pazocal{A}-\mathrm{injective}$.
\end{example}
\begin{proof}
Let $\{S_{i}\}_{i\in I}$ be the irredundant set of representatives of all simple $R-\mathrm{modules}$ and $E$ be the injective envelope of $\oplus_{i\in I} S_{i}$. Then $E$ is an injective  cogenerator and $Hom(S,E)\cong S$ by Corollary 18.19 in \cite{13} and Lemma 2.6 in \cite{14}, respectively. There is an isomorphism as follow $$Ext^1_{R}(M/A,Hom(S,E))\cong Hom(Tor_{1}^{R}(M/A,S),E).$$ So the requireds follow.
\end{proof}
\begin{lem} \label{lem18}
\begin{itemize}
\item[(i)] $(S) M-\pazocal{A}F$ is closed under pure submodules and pure quotient modules.
\item[(ii)] $(S) M-\pazocal{A}I$ is closed under pure submodules and pure quotient modules if $M$ is finitely presented and $\pazocal{A}-\mathrm{coherent}$.
\end{itemize}
\end{lem}
\begin{proof}
\begin{itemize}
\item[(i)] Let $N$ be pure submodule of an $(S) M-\pazocal{A}-\mathrm{flat}$ module $T$. Then the exact sequence $0\rightarrow N\rightarrow T\rightarrow T/N\rightarrow0$ induces the split exact sequence $0\rightarrow(T/N)^+\rightarrow T^+\rightarrow N^+\rightarrow0$. By Lemma ~\ref{lem11}, $T^+$ is $(S) M-\pazocal{A}-\mathrm{injective}$ and then $(T/N)^+$ and $N^+$ are also $(S) M-\pazocal{A}-\mathrm{injective}$. By Lemma ~\ref{lem11} again, $T/N$ and $N$ are $(S) M-\pazocal{A}-\mathrm{flat}$.
\item[(ii)] The proof is similar to the part (i).
\end{itemize}
\end{proof}
\begin{rem}
Let $M$ be (finitely presented and) $\pazocal{A}-\mathrm{coherent}$. Then $(S) M-\pazocal{A}I$ is closed under direct sum.
\end{rem}
We recall the following definitions (see \cite{8}). A cotorsion theory $(\pazocal{A},\pazocal{B})$ is called perfect if every module has a $\pazocal{B}-\mathrm{envelope}$ and $\pazocal{A}-\mathrm{cover}$. A cotorsion theory $(\pazocal{A},\pazocal{B})$ is complete if for any module $N$, there are exact sequences $0\rightarrow N\rightarrow B\rightarrow A\rightarrow0$ where $B\in\pazocal{B}$ and $A\in\pazocal{A}$, and $0\rightarrow B_{1}\rightarrow A_{1}\rightarrow N\rightarrow0$ where $B_{1}\in\pazocal{B}$ and $A_{1}\in\pazocal{A}$. Now we will give the following Theorem.
\begin{thm}
\begin{itemize}
\item[(i)]$(S) M-\pazocal{A}F$ is a Kaplansky class and also if $M$ is finitely presented and $\pazocal{A}-\mathrm{coherent}$, then $(S) M-\pazocal{A}I$ is a Kaplansky class.
\item[(ii)] Every left $R-\mathrm{module}$ has an $(S) M-\pazocal{A}I$ preenvelope where $M$ is finitely presented and $\pazocal{A}-\mathrm{coherent}$.
\item[(iii)] Every left $R-\mathrm{module}$ has an $(S) M-\pazocal{A}I$ cover where $M$ is finitely presented and $\pazocal{A}-\mathrm{coherent}$.
\item[(iv)] $((S) M-\pazocal{A}F(I),((S) M-\pazocal{A}F(I))^\perp)$ is perfect cotorsion theory (where $M$ is finitely presented, $\pazocal{A}-\mathrm{coherent}$ and $R$ is in $(S) M-\pazocal{A}I$).
\item[(v)] The cotorsion theory $(^{\perp}(S M-\pazocal{A}I),S M-\pazocal{A}I)$ is complete.

\end{itemize}
\end{thm}
\begin{proof}
\begin{itemize}
\item[(i)]It follows from Lemma 5.3.12 in \cite{8} and Lemma \ref{lem18}.
\item[(ii)] By Proposition 6.2.1 in \cite{8}, it is understood.
\item[(iii)] Every left $R-\mathrm{module}$ has an $(S) M-\pazocal{A}I$ cover by Theorem 2.5 in \cite{11} and Lemma ~\ref{lem18}.
\item[(iv)] $((S) M-\pazocal{A}F(I),((S) M-\pazocal{A}F(I))^\perp)$ is a perfect cotorsion theory by Theorem 3.4 in \cite{11}.
\item[(v)] Let $C$ be the set of representatives of all the $M/A$'s where $A\in\pazocal{A}$. Then $S M-\pazocal{A}I=C^\bot$. By Theorem 10 in \cite{12}, it is a complete cotorsion theory.
\end{itemize}
\end{proof}

\begin{thm}\label{thm21}
Let $M$ be finitely presented. Then the followings are equivalent:
\begin{itemize}
\item[(1)] $M$ is $\pazocal{A}-\mathrm{coherent}$.
\item[(2)] $(S) M-\pazocal{A}I$ is closed under direct limit.
\item[(3)] $(S) M-\pazocal{A}F$ is closed under direct product (where $R$ is a coherent ring).
\item[(4)] A left $R-\mathrm{module}$ $T$ is in $(S) M-\pazocal{A}I$ if and only if $T^+$ is in $(S) M-\pazocal{A}F$.
\item[(5)] A right $R-\mathrm{module}$ $N$ is in $(S) M-\pazocal{A}F$ if and only if $N^{++}$ is in $(S) M-\pazocal{A}F$.
\item[(6)] Every right $R-\mathrm{module}$ has an $(S) M-\pazocal{A}F$ preenvelope.
\end{itemize}
\end{thm}
\begin{proof}
(1)$\Rightarrow$(2) Let $\{N_{i}\}_{i\in I}$ be any family of $(S) M-\pazocal{A}-\mathrm{injective}$ modules. Since $(S) M-\pazocal{A}I$ is closed under pure quotient modules, by the pure exact sequence $0\rightarrow Y\rightarrow \oplus N_{i}\rightarrow \underaccent{\rightarrow}{\lim} N_{i}\rightarrow0$ by 33.9 in \cite{7}, we have that $\underaccent{\rightarrow}{\lim} N_{i}$ is in $(S) M-\pazocal{A}I$.\par
(2)$\Rightarrow$(1) We have the following commutative diagram
\[\begin{diagram}
\node{\underaccent{\rightarrow}{\lim} Hom(M,N_{i})} \arrow{e,t}{\alpha}\arrow{s,t}{\beta} \node{\underaccent{\rightarrow}{\lim} Hom(A,N_{i})} \arrow{s,t}{\theta} \\
\node{Hom(M,\underaccent{\rightarrow}{\lim}N_{i})} \arrow{e,t}{\gamma} \node{Hom(A,\underaccent{\rightarrow}{\lim}N_{i})}
\end{diagram}\]
where $\theta$ is monic by 24.9 in \cite{7}.\par
By 25.4 in \cite{7}, $\beta$ is an isomorphism. Since $\gamma$ is epic, $\theta$ is epic. By 25.4 in \cite{7} again, $A$ is finitely presented.\par
(1)$\Rightarrow$(3) Let $\{N_{i}\}_{i\in I}$ be a family of $M-\pazocal{A}-\mathrm{flat}$ modules. We have the following commutative diagram
\[\begin{diagram}
\node{(\Pi N_{i})\otimes A} \arrow{e,t}{\alpha}\arrow{s,t}{\beta} \node{(\Pi N_{i})\otimes M} \arrow{s,t}{\theta} \\
\node{\Pi (N_{i}\otimes A)} \arrow{e,t}{\gamma} \node{\Pi (N_{i}\otimes M)}
\end{diagram}\]\par
By Theorem 3.2.22 in \cite{8}, $\beta$ is an isomorphism. Since $\gamma$ is one to one and $\alpha$ is one to one. Thus $\Pi N_{i}$ is in $M-\pazocal{A}I$. (Notice that it is not necessary that $M$ is finitely presented.)\par
Now, we will show that if $\{N_{i}\}_{i\in I}$ is a family of $SM-\pazocal{A}-\mathrm{flat}$ modules, then $\Pi N_{i}$ is in $SM-\pazocal{A}F$.\par
By Theorem 3.2.26 in \cite{8} $$Tor_{1}^{R}(\Pi N_{i}, M/A)\cong \Pi Tor_{1}^{R}( N_{i}, M/A)$$ where $R$ is coherent and $M/A$ is finitely presented by Lemma ~\ref{lem13}. So the required is found.\par
(3)$\Rightarrow$(1) Since $R$ is flat, $\Pi R$ is in $(S) M-\pazocal{A}F$. We have the commutative diagram
\[\begin{diagram}
\node{\Pi R\otimes A} \arrow{e,t}{\alpha}\arrow{s,t}{\beta} \node{\Pi R\otimes M} \arrow{s,t}{\theta} \\
\node{\Pi A} \arrow{e,t}{\gamma} \node{\Pi M}
\end{diagram}\]
where $\beta$ is epic by Lemma 3.2.21 in \cite{8}.\par
By Theorem 3.2.22 in \cite{8}, $\theta$ is isomorphism, so $\beta$ is isomorphism. So $\pazocal{A}$ is finitely presented by Theorem 3.22 in \cite{8}.\par
(1)$\Rightarrow$(4) follows from Lemma ~\ref{lem14}.\par
(4)$\Rightarrow$(5) is easy.\par
(5)$\Rightarrow$(3) Let $\{N_{i}\}_{i\in I}$ be a family of $(S) M-\pazocal{A}-\mathrm{flat}$ right $R-\mathrm{modules}$. Then $\oplus N_{i}$ is also in $(S) M-\pazocal{A}F$. By the equivalence $(\oplus N_{i})^{++}\cong (\Pi N_{i}^{+})^+$ and the part (6), $(\Pi N_{i}^{+})^+$ is in $(S) M-\pazocal{A}F$. Since $\oplus N_{i}^+$ is a pure submodule of $\Pi N_{i}^+$ by Lemma 1 (1) in \cite{9}, so $(\Pi N_{i}^{+})^+\rightarrow (\oplus N_{i}^{+})^+\rightarrow0$ is split. Therefore $(\oplus N_{i}^{+})^+$ is in $(S) M-\pazocal{A}F$. Since $\Pi N_{i}^{++}\cong (\oplus N_{i}^+)^+$, $\Pi N_{i}^{++}$ is in $(S) M-\pazocal{A}F$. Since $\Pi N_{i}$ is a pure submodule of $\Pi N_{i}^{++}$ by Lemma 1 (2) in \cite{9}, so $\Pi N_{i}$ is in $(S) M-\pazocal{A}F$.\par
(3)$\Leftrightarrow$(6) follows from Theorem 3.3 in \cite{10} and Lemma ~\ref{lem18}.
\end{proof}

\section{$F\pazocal{A}$ and $P\pazocal{A}$ modules}
\begin{defn}
If $A$ is flat \emph{(}projective\emph{)} for all $A\in\pazocal{A}$, then $M$ is called an $F\pazocal{A}(P\pazocal{A})$- module. If $M$ is $\pazocal{A}$-coherent, then $M$ is an $F\pazocal{A}$-module if and only if $M$ is a $P\pazocal{A}$- module.
\end{defn}
\begin{prop}\label{prop3.2}
The following are satisfied:\par
\begin{itemize}
\item[(1)] $M$ is an $F\pazocal{A}-\mathrm{module}$.
\item[(2)] Every submodule of $M-\pazocal{A}-flat$ module is again $M-\pazocal{A}-flat$.
\item[(3)] Every right $R-\mathrm{module}$ has an epic $(S) M-\pazocal{A}-\mathrm{flat}$ preenvelope where $M$ is $\pazocal{A}-\mathrm{coherent}$ \emph{(}and finitely presented and $R$ is coherent\emph{)}.
\end{itemize}
Then \emph{(1)}$\Rightarrow$\emph{(2)}. If $M$ is flat, then \emph{(2)}$\Rightarrow$\emph{(1)}. If $M$ is $\pazocal{A}-\mathrm{coherent}$ and flat, then \emph{(1)}$\Leftrightarrow$\emph{(2)}$\Leftrightarrow$\emph{(3)}.
\end{prop}
\begin{proof}
(1)$\Rightarrow$(2) Let $B$ be any $M-\pazocal{A}-\mathrm{flat}$ module and $C$ be a submodule of $B$. We have to the following commutative diagram \par
\[\begin{diagram}
\node{C\otimes A} \arrow{e,t}{\alpha}\arrow{s,t}{\beta} \node{C\otimes M} \arrow{s,t}{\theta} \\
\node{B\otimes A} \arrow{e,t}{\gamma} \node{B\otimes M}
\end{diagram}\]
where $\beta$ and $\gamma$ are one to one. So $\alpha$ is one to one, and the required follows.\par
(2)$\Rightarrow$(1) Since $M$ is flat, any $M-\pazocal{A}-\mathrm{flat}$ is $S M-\pazocal{A}-\mathrm{flat}$. We have the exact sequence $0=Tor_{2}^{R}(R,M/A)\rightarrow Tor_{2}^{R}(R/I,M/A)\rightarrow Tor_{1}^{R}(I,M/A)=0$. So $Tor_{2}^{R/I}(R,M/A)=0$. Since we have the exact sequence $0=Tor_{2}^{R}(R/I,M/A)\rightarrow Tor_{1}^{R}(R/I,A)\rightarrow Tor_{1}^{R}(R/I,M)=0$, so $Tor_{1}^{R}(R/I,A)=0$, and hence $A$ is flat.\par
(2)$\Rightarrow$(3) By Theorem ~\ref{thm21}, every right $R-\mathrm{module}$ $B$ has an $(S) M-\pazocal{A}-\mathrm{flat}$ preenvelope $\theta:B\rightarrow C$. By the part (2) $\theta:B\rightarrow Im\theta$ is an epic $(S) M-\pazocal{A}F$ preenvelope of $B$.\par
(3)$\Rightarrow$(2) Let $C$ be any right $R-\mathrm{submodule}$ of an $(S) M-\pazocal{A}-\mathrm{flat}$ module $B$. Then $C$ has an epic $(S) M-\pazocal{A}-\mathrm{flat}$ preenvelope $\theta:C\rightarrow D$. So we have the following diagram\par
\[\begin{diagram}
\node[2]{B}\\
\node{C} \arrow{ne,t} {i} \arrow{e,t} {\theta} \node{D} \arrow{n,t,..} {\gamma} \arrow{e} \node{0}
\end{diagram}\]
where $\gamma\theta=i$. So $\theta$ is isomorphism.
\end{proof}
\begin{prop}
The followings are satisfied.
\begin{itemize}
\item[(1)] $M$ is a $P\pazocal{A}-\mathrm{module}$.
\item[(2)] Every quotient module of any $(S) M-\pazocal{A}-\mathrm{injective}$ left $R-\mathrm{module}$ is $(S) M-\pazocal{A}-\mathrm{injective}$(where $M$ is projective).
\item[(3)] $M$ is $\pazocal{A}-\mathrm{coherent}$ and every submodule of any $M-\pazocal{A}-\mathrm{flat}$ is $M-\pazocal{A}-\mathrm{flat}$.
\item[(4)] Every left $R$-module has a monic $(S) M-\pazocal{AI}$ cover.
\end{itemize}
Then \emph{(1)}$\Rightarrow$\emph{(2)}. If $M$ is projective, then \emph{(2)}$\Rightarrow$\emph{(1)} and \emph{(4)}$\Rightarrow$\emph{(2)}. If $M$ is finitely presented, then \emph{(3)}$\Rightarrow$\emph{(2)}. If $A$ is finitely generated for all $A\in \pazocal{A}$ and $M$ is projective, then \emph{(2)}$\Rightarrow$\emph{(3)}. If $M$ is $\pazocal{A}-\mathrm{coherent}$, then \emph{(2)}$\Rightarrow$\emph{(4)}.
\end{prop}
\begin{proof}
(1)$\Rightarrow$(2) Let $B$ be any $M-\pazocal{A}-\mathrm{injective}$ left $R-\mathrm{module}$ and $C$ be any submodule of $B$. Since $\Pi:B\rightarrow B/C$ is the canonical map and $A$ is projective, for any $f:A\rightarrow B/C$, there exits $g:A\rightarrow B$ such that $\Pi g=f$. So there exits $h:M\rightarrow B$ such that $hi=g$ where $i:A\rightarrow M$ is the inclusion. Then $(\Pi h)i=f$ and then (2) holds.\par
(2)$\Rightarrow$(1) Since $M$ is projective, $S M-\pazocal{A}I=M-\pazocal{A}I$. By the exact sequence $0\rightarrow N\rightarrow E\rightarrow E/N\rightarrow0$ where $N$ is any left $R-\mathrm{module}$ and $E$ is the injective envelope of $N$, we have the exact sequence $0=Ext^{1}_{R}(M/A,E/N)\rightarrow Ext^{2}_{R}(M/A,N)\rightarrow Ext^{2}_{R}(M/A,E)=0$ and then $Ext^{2}_{R}(M/A,N)=0$.\par
By the exact sequence $0\rightarrow A\rightarrow M\rightarrow M/A\rightarrow0$, we have that $0=Ext^{1}_{R}(M,N)\rightarrow Ext^{1}_{R}(A,N)\rightarrow Ext^{2}_{R}(M/A,N)=0$. So $Ext^{1}_{R}(A,N)=0$, and then $A$ is projective.\par
(3)$\Rightarrow$(2) Let $B$ be any $M-\pazocal{A}-\mathrm{injective}$ module and $C$ be a submodule of $B$. Then the exact sequence $0\rightarrow C\rightarrow B\rightarrow B/C\rightarrow0$ induces the exact sequence $0\rightarrow (B/C)^+\rightarrow B^+\rightarrow C^+\rightarrow0$. Since $B^+$ is $M-\pazocal{A}-\mathrm{flat}$ Lemma ~\ref{lem14} and so $(B/C)^+$ is $M-\pazocal{A}-\mathrm{flat}$. So $B/C$ is $M-\pazocal{A}-\mathrm{injective}$.\par
(2)$\Rightarrow$(3) Since (2)$\Rightarrow$(1), by Proposition ~\ref{prop3.2} every submodule of $M-\pazocal{A}-\mathrm{flat}$ module is $M-\pazocal{A}-\mathrm{flat}$. Let $N$ be an $FP-\mathrm{injective}$ left $R-\mathrm{module}$ and $E$ be the injective envelope of $N$. Then $E/N$ is $M-\pazocal{A}-\mathrm{injective}$. So we have the exact sequence $0=Ext^{1}_{R}(M/A,E/N)\rightarrow Ext^{2}_{R}(M/A,N)\rightarrow Ext^{2}_{R}(M/A,E)=0$ and so $Ext^{2}_{R}(M/A,N)=0$.\par
By the exact sequence $0\rightarrow A\rightarrow M\rightarrow M/A\rightarrow0$, we have the exact sequence $0=Ext^{1}_{R}(M,N)\rightarrow Ext^{1}_{R}(A,N)\rightarrow Ext^{2}_{R}(M/A,N)=0$, and so $Ext^{1}_{R}(A,N)=0$, by \cite{15} $A$ is finitely presented.\par
(2)$\Rightarrow$(4) Let $N$ be a left $R-\mathrm{module}$. Let $F=\sum\{x\leq N: x\in M-\pazocal{A}I\}$ and $G=\oplus\{x\leq N: x\in M-\pazocal{A}I\}$. Then there exists an exact sequence $0\rightarrow K\rightarrow G\rightarrow F\rightarrow0$ by \cite{16}. Note that $G\in M-\pazocal{A}I$, so $F\in M-\pazocal{A}I$ by (2).\par
Let $\psi: F^{1}\rightarrow N$ with $F^{1}\in M-\pazocal{A}I$ be any left $R-\mathrm{homomorphism}$. By (2), $\psi(F^{1})\leq F$. Let $\theta:F^1\rightarrow F$ with $\theta{x}=\psi{x}$ for $x\in F^1$. Then $i\theta=\psi$ and so $i:F\rightarrow N$ is an $M-\pazocal{A}I$ precover of $N$. Moreover the identity map $I_{F}$ of $F$ is the only homomorphism $g:F\rightarrow F$ such that $ig=i$, and so (4) holds.\par
(4)$\Rightarrow$(2) Let $B$ be any $M-\pazocal{A}-\mathrm{injective}$ module and $C$ be a submodule of $B$. If $E$ is injective envelope of $C$ and $\phi:F\rightarrow E/C$ is the monic $M-\pazocal{A}I$ cover, then we have the following diagram\par
\[\begin{diagram}
\node[4]{0} \arrow{s} \\
\node[4]{F} \arrow{s,t} {\phi} \\
\node{0} \arrow{e} \node{C} \arrow{e} \node{E} \arrow{ne,t,..} {\alpha} \arrow{e} \node{E/C} \arrow{e} \node{0}
\end{diagram}\]
Then $\phi$ is epic, so it is isomorphism. Therefore $E/C$ is $M-\pazocal{A}-\mathrm{injective}$. We have the exact sequence $0=Ext^{1}_{R}(M/A,E/C)\rightarrow Ext^{2}_{R}(M/A,C)\rightarrow Ext^{2}_{R}(M/A,E)=0$,  and so $Ext^{2}_{R}(M/A,C)=0$. By the exact sequence $0=Ext^{1}_{R}(M/A,B)\rightarrow Ext^{1}_{R}(M/A,B/C)\rightarrow Ext^{2}_{R}(M/A,C)=0$, and so $Ext^{1}_{R}(M/A,B/C)=0$, so (2) holds.
\end{proof}
Using Corollory \ref{cor16} we can give the following proposition:
\begin{prop} The following are equivalent:
\begin{itemize}
\item[(i)] $A$ is $M-\pazocal{A}-\mathrm{injective}$ for all $A\in \pazocal{A} $.
\item[(ii)] $A$ is a direct summand of $M$ for all $A\in \pazocal{A} $.
\item[(iii)] Every left $R-\mathrm{module}$ is $M-\pazocal{A}-\mathrm{injective}$.
\item[(iv)] Every right $R-\mathrm{module}$ is $M-\pazocal{A}-\mathrm{flat}$ where $M$ is finitely presented and $\pazocal{A}-\mathrm{coherent}$.
\end{itemize}
\end{prop}
\begin{thm}The following are equivalent:
\begin{itemize}
\item[(i)]Every cotorsion left $R-\mathrm{module}$ is in $SM-\pazocal{A}I$.
\item[(ii)] Every pure injective left $R-\mathrm{module}$ is in $SM-\pazocal{A}I$.
\item[(iii)]$M/A$ is flat for all $A\in \pazocal{A} $.
\item[(iv)]Every right $R-\mathrm{module}$ is in $SM-\pazocal{A}F$.
\item[(v)]$Tor_{1}^{R}(R/I,M/A)=0$ where $I$ is any right ideal of $R$ and for all $A\in \pazocal{A} $.
\item[(vi)]$A\cap IM=IA$ for all rigth ideal $I$ of $R$ where $M$ is flat.
\end{itemize}

\end{thm}
\begin{proof}(i)$\Rightarrow$(ii), (iii)$\Rightarrow$(i), (iii)$\Leftrightarrow$(iv)$\Leftrightarrow$ (v) are trivial.\\
(ii)$\Rightarrow$(iii) Let $N$be any rigth $R-\mathrm{module}$. Then $N^+$ is pure injective. Since $0=Ext^{1}_{R}(M/A,N^+)\cong Hom_{\mathbb {Z}}(Tor_{1}^{R}(N,M/A),\mathbb{Q}/\mathbb{Z}),\\Tor_{1}^{R}(N,M/A)=0$, so $M/A$ is flat.\\
(v)$\Leftrightarrow$(vi) We have the following commutative diagram:
\[\begin{diagram}
\node{0} \arrow{e} \node{Tor_{1}^{R}(R/I,M/A)} \arrow{e} \arrow{s,t}{\alpha} \node{I\otimes M/A} \arrow{e} \arrow{s,t}{\beta} \node{R\otimes M/A}\arrow{s,t}{\gamma} \\
\node{0} \arrow{e} \node{\frac{IM\cap A}{IA}} \arrow{e} \node{IM/IA} \arrow{e} \node{M/A}
\end{diagram}\]\par
Since $\beta$ and $\gamma$ are isomorphisms, so is $\alpha$.
\end{proof}

In the following proposition it is also satisfied all equivalent conditions in Theorem 3.5.
\begin{prop} Let $R$ be in $(S) M-\pazocal{A}I$. If $M$ is in $P\pazocal{A}$ and $M$ is flat, then $M/A$ is flat for all $A\in \pazocal{A} $.
\end{prop}
\begin{proof} Let us given that $a_j \in A$, $m_i\in M$ and $s_{ij}\in R$ such that $a_j=\sum\limits_{i=1}^{m} s_{ij}m_i$ where $1\leq j\leq n$. By the Dual Basis Lemma, there exist $\{f_k\;:\;k\in I\}\subseteq Hom(A,R)$ such that for any $c\in A$ $f_k(c)=0$ for almost all $k$, and $c=\sum f_k(c)c_k$. Since $R$ is in $(S) M-\pazocal{A}I$, there exist $g_k\in Hom(M,R)$ such that $f_k(a_j)=g_k(a_j)=g_k(\sum\limits_{i=1}^{m} s_{ij}m_i)=\sum s_{ij}g_k(m_i)$. So $a_j=\sum f_k(a_j)c_k=\sum s_{ij}(\sum g_k(m_i)c_k)$. So $A$ is pure submodule of $M$ by Theorem 4.89 in \cite{17}.

\end{proof}
\begin{rem}Let $\pazocal{A}$ be any family of some left $R-$ modules. Then A left $R-$module $T$ is called a strongly $\pazocal{A}-\mathrm{injective}(flat)$ if for every exact sequence $0\rightarrow X\rightarrow Y\rightarrow Y/X\rightarrow0$  where $Y/X\cong S$ for some $S\in\pazocal{A}$, we have an exact sequence $Hom(Y,T)\rightarrow Hom(X,T)\rightarrow0$ ($(T\otimes Y)^{+}\rightarrow(T\otimes X)^{+}\rightarrow0$), or equivalently $Ext^{1}_{R}(S,T)=0$ ($Tor_{1}^{R}(T,S)=0$) for all $S\in\pazocal{A}$. The class of strongly $\pazocal{A}-\mathrm{injective}(flat)$ modules is denoted by $S-\pazocal{A}I(F)$ and hence $S-\pazocal{A}I=\pazocal{A}^{\perp}$ and $S-\pazocal{A}F=^{\top}\pazocal{A}$ (see for notations \cite{22}). In this case $SM-\pazocal{A}I(F)$ is a special case of $S-\pazocal{A}I(F)$.

From the proof above we see the following:
\begin{itemize}
\item[(i)] $S-\pazocal{A}I(F)$ is closed under extensions, direct summand and direct product ( extensions, direct summand, direct sum and direct limit).
\item[(ii)] A right  $R-$ module $N$ is in $S-\pazocal{A}F$ if and only if $N^+$ is in $S-\pazocal{A}I$.
\item[(iii)] Let the modules in $\pazocal{A}$ be finitely presented then left $R-$ module $N$ is in $S-\pazocal{A}I$ 
if and only if $N^+$ is in $S-\pazocal{A}F$.
\item[(iv)] A simple module over commutative ring is in $S-\pazocal{A}F$ if and only if it is in $S-\pazocal{A}I$.
\item[(v)] $S-\pazocal{A}F(I)$ is closed under pure submodule and pure quotient module and by Lemma 5.3.12 in \cite{8} it is a Kaplansky class  (where the modules in $\pazocal{A}$ are finitely presented ).
\item[(vi)] $S-\pazocal{A}I$ is preenveloping  and covering class where the modules in $\pazocal{A}$ are finitely presented.
\item[(vii)] $(S-\pazocal{A}F(I),S-\pazocal{A}F(I)^{\perp})$ is a perfect cotorsion theory (where the modules in $\pazocal{A}$ be finitely presented  and $R$ is in $S-\pazocal{A}I$). 
\item[(viii)] $(^\perp(S-\pazocal{A}I),S-\pazocal{A}I)$ is a complete cotorsion theory.
\item[(ix)] $S-\pazocal{A}I$ is closed under direct sum and direct limit if the modules of $\pazocal{A}$ are finitely presented.
\item[(x)] $S-\pazocal{A}F$ is closed under direct product if the modules of $\pazocal{A}$ are finitely presented and $R$ is a coherent ring.
\item[(xi)] $S-\pazocal{A}F$ is closed under direct product if and only if $S-\pazocal{A}F$ is preenveloping class.

\end{itemize} 
\end{rem}

\end{document}